\documentclass{article}
\usepackage[utf8]{inputenc}
\usepackage[T1]{fontenc}
\usepackage[english]{babel}

\usepackage{amssymb}
\usepackage{verbatim}
\usepackage{array}
\usepackage{latexsym}
\usepackage{enumerate}
\usepackage{amsmath}
\usepackage{amsfonts}
\usepackage{amsthm}
\usepackage{color}
\usepackage{graphicx}
\usepackage{dsfont}
\usepackage{float}
\usepackage{hyperref}

\hypersetup{
    colorlinks=true,
    linkcolor=green,
    citecolor=green,
    filecolor=cyan,      
    urlcolor=cyan,
}

\newtheorem{theorem}{Theorem}[section]

\newtheorem{lemma}[theorem]{Lemma}
\newtheorem{corollary}[theorem]{Corollary}

\newtheorem{example}[theorem]{Example}
\newtheorem{remark}[theorem]{Remark}

\newcommand{\xx}{\mathcal X}

\newcommand{\cC}{\mathcal C}

\newcommand{\cL}{\mathcal L}

\newcommand{\F}{\mathbb F}
\newcommand{\N}{\mathbb N}
\newcommand{\Z}{\mathbb Z}

\newcommand{\K}{\mathbb K}

\newcommand{\ga}{\alpha}
\newcommand{\gb}{\beta}
\newcommand{\gs}{\sigma}

\newcommand{\gO}{\Omega}

\title{Weierstrass pure gaps on curves with three distinguished points}
\date{}
\author{Herivelto Borges \\ Gregory Cunha}
\begin{document}
\maketitle

\begin{abstract}
Let  $\K$ be an algebraically closed field. In this paper, we  consider the class of  smooth plane curves of degree $n+1>3$ over $\K$,    containing  three points,  $P_1,P_2,$ and $P_3$, such that $nP_1+P_2$, $nP_2+P_3$, and $nP_3+P_1$ are divisors cut out by three distinct lines. For such curves, we determine the dimension of certain special divisors  supported on $\{P_1,P_2,P_3\}$, as well as  an explicit description of  all  pure gaps at any subset of $\{P_1,P_2,P_3\}$. When $\K=\overline{\F}_q$, this class of  curves, which includes the Hermitian curve,  is  used to construct algebraic geometry  codes having minimum distance better than the Goppa bound.
\end{abstract}

\providecommand{\keywords}[1]{\textbf{\textit{Keywords:}} #1}

\\


\section{Introduction}

\indent

Let $\xx$ be a projective, geometrically irreducible, nonsingular algebraic curve of genus $g$ defined over an algebraically closed field $\K$, and let $\K(\xx)$ be its field of rational functions. 
For distinct points $P_1, \ldots, P_r \in \xx$, the  Weierstrass semigroup of $\xx$ at $(P_1, \dots, P_r)$ is defined by
$$
H(P_1, \dots, P_r) := \left\{ (n_1, \dots, n_r) \in \N_0^r : \; \exists f \in \K(\xx) \text{ with } \mbox{div}_{\infty}(f) = \sum_{i=1}^{r} n_i P_i \right\},
$$
where $\N_0$ denotes the set of nonnegative integers, and $\mbox{div}_{\infty}(f)$ denotes the pole divisor of $f$.
The elements in the finite complement 
$G(P_1, \ldots, P_r) := \N_0^r \backslash H(P_1, \dots, P_r)$
are called  Weierstrass gaps of $\xx$ at $(P_1, \dots, P_r)$. An $r$-tuple $(n_1, \dots, n_r) \in \N_0^r$ is a  pure gap of $\xx$ at $(P_1, \dots, P_r)$ if the Riemann-Roch spaces
$\cL \left( \sum_{i=1}^{r} n_i P_i \right)$ and $\cL \left( \sum_{i=1}^{r} n_i P_i - P_j \right)$
have the same dimension for all $j=1,\dots,r$.
The set of pure gaps at $(P_1, \dots, P_r)$ is denoted by $G_0(P_1, \dots, P_r)$. It is easy to check that $G_0(P_1, \dots, P_r) \subset G(P_1, \dots, P_r)$.
\\

The notion of Weierstrass semigroups at several points was introduced by Arbarello, Cornalba, Griffiths, and Harris \cite[p. 365]{Arbarello}, and the concept of pure gaps of a pair of points on a curve was first presented and investigated by Homma and Kim \cite{HK}. The latter study was extended to several points by Carvalho and Torres \cite{CT}.
After these two seminal papers, many authors have pursuit the characterization of Weierstrass semigroups and pure gaps on special families of curves  \cite{bartoli2018}, \cite{cicero-kato}, \cite{matthews}, \cite{TT}. This problem, which involves  determining the dimension of certain divisors,  is challenging and important in its own right. It  can be related to many other problems,  such as  bounding the number of rational points on curves over finite fields \cite{Beelen}, \cite{Stohr-Voloch}.  Its study  can be further motivated by the construction of algebraic geometry codes, also known as Goppa codes, with good parameters. In fact, from \cite{HK} and \cite{CT}, one can see that an explicit description of pure gaps at several points can be used to construct Goppa codes with large minimum distance. 
\\

In this article, we   consider the following type of curves.

\begin{equation}\label{curvaX}
\fbox{\begin{minipage}{38em}
$\xx$ is a smooth plane curve of degree $n+1>3$, equipped with points  $P_1,P_2,$ and $P_3$,  and three distinct lines cutting out on $\xx$ the divisors $nP_1+P_2$, $nP_2+P_3$, and $nP_3+P_1$.
\end{minipage}}
\end{equation}

The Hermitian curve $\mathcal{H}\, : Y^qZ+YZ^q-X^{q+1}=0$ over $\overline{\F}_q$, $q>2$, is a well-known example of curve of type \eqref{curvaX}. In fact, for any of the  points $P \in \mathcal{H}(\F_{q^6})$
which are not defined over $\F_{q^2}$, the Frobenius images
$$P,P^{q^2},P^{q^4} \in  \mathcal{H}$$
are points with the desired property. A more detailed characterization  of  the   curves  \eqref{curvaX} is will be provided  in Section \ref{sec3}.
For such a family of curves,  we provide   an explicit description of the set of pure gaps at any subset of $\{P_1,P_2,P_3\}$. More precisely,  we prove the  following main results.

\begin{theorem}\label{p1p2}
Let $\xx$ be a curve given by \eqref{curvaX}, and let $g$ be its genus. The set $G_0(P_1,P_2)$ of pure gaps of $\xx$ at $(P_1,P_2)$ has cardinality $\frac{1}{3}(g-1)g$, and it is comprised by the elements
$$
(a,b) := \Big((i-1)n+r+1,(j-1)n+i+s \Big),
$$
with $i,j \geq 1$, $ i+j=d \leq n-1$, $r \in \{0,\ldots, n-d-1\}$ and $s \in \{0,\ldots, n-d\}$.
Moreover, the Riemann-Roch spaces $\cL(aP_1+bP_2)$ and $\cL((a-1)P_1+(b-1)P_2)$ have dimension $\frac{1}{2}(d-1)(d-2) + i$, and the same holds for the ordered pairs $(P_2,P_3)$ and $(P_3,P_1)$.
\end{theorem}

\begin{theorem}\label{cor}
Let $\xx$ be a curve given by \eqref{curvaX}. The set of pure gaps $G_0(P_1,P_2,P_3)$ of $\xx$ at $(P_1,P_2,P_3)$ is comprised by the elements
$$
(a,b,c):=\Big(kn+j+r+1,in+k+s+1, jn+i+t+1\Big),
$$
with $i,j,k \geq 0$, $i+j+k=d \leq n-3$ and $r,s,t \in \{0,\ldots, n-3-d\}$.
Moreover, the Riemann-Roch spaces $\cL(aP_1+bP_2+cP_3)$ and $\cL((a-1)P_1+(b-1)P_2+(c-1)P_3)$ have dimension $\frac{1}{2}(d+1)(d+2)$. In particular, the number of pure gaps at $(P_1,P_2,P_3)$ is
$$
\# G_0(P_1,P_2,P_3) = \frac{1}{30}(g-1)g(2g-1),
$$
where $g$ is the genus of $\xx$.
\end{theorem}

This paper is organized as follows. In Section \ref{sec2}, we  establish notation and collect results that will be used throughout the paper. In Section \ref{sec3},  we prove Theorem \ref{gap1}  which provides the Weierstrass semigroup and the gap sequence of $\xx$ at any point $P \in \{P_1,P_2,P_3\}$. In Section \ref{sec4}, we  determine  the dimension of certain Riemann-Roch spaces and explicitly describe  the set of pure gaps of $\xx$ at $(P_1,P_2)$.  In Section \ref{sec5}, we  solve the problem for the case of pure gaps at $(P_1,P_2,P_3)$. In Section \ref{sec6},  we apply the  previous results to construct Goppa codes with good parameters, and  we present particular codes establishing  new records in the Mint database \cite{MinT}.

\section{Preliminaries} \label{sec2}

\indent

Let $\xx$ be a projective, geometrically irreducible, nonsingular algebraic curve of genus $g$ defined over an algebraically closed field $\K$, and let $(P_1, \ldots, P_r)$ be an $r$-tuple of distinct points on $\xx$. For a divisor $D$ on $\xx$, the {\it Riemann-Roch space} associated to $D$ is $\cL(D) := \{ f \in \K(\xx) : \mbox{div}(f) \geq - D \} \cup \{0\}$, and its dimension is denoted by $\ell(D)$.
\\

For $\ga = (\ga_1,\dots,\ga_r) \in \N_0^r$ and $i \in \{1,\dots,r\}$, set
$$
\nabla^r_i(\ga) := \{(\gb_1,\dots,\gb_r) \in H(P_1,\dots,P_r) : \gb_i=\ga_i \ \text{and} \ \gb_j \leq \ga_j \ \forall j \neq i\}.
$$

The following result was proven by Carvalho and Torres in \cite{CT} \footnote{Their original result  assumes that the ground field $\F$ is a perfect field with $\#\F \geq r$. We consider  algebraically closed fields  for the sake of simplicity.}.

\begin{theorem}\label{thm_faces}
Let $\ga = (\ga_1, \ldots, \ga_r) \in \N_0^r$. Then the following are equivalent.
\begin{enumerate}[\rm(1)]
\item $\ga \in G_0(P_1, \dots, P_r)$
\item $\nabla^r_i(\ga) = \emptyset$, for all $i=1,\ldots,r$
\item $\ell (\sum_{i=1}^r \ga_iP_i) = \ell(\sum_{i=1}^r (\ga_i-1)P_i)$.
\end{enumerate}
Moreover, if $\ga \in G_0(P_1, \dots, P_r)$ then $\ga_i \in G(P_i)$ for all $i=1,\ldots,r$.
\end{theorem}

Now, suppose that $r=2$. For a gap $a$ at $P_1$,  define $\gb_a := \min \{ t  :  (a,t) \in H(P_1, P_2) \}$. In \cite[Lemma 2.6]{kim}, Kim proved that $\{  \gb_a : a \in G(P_1) \}  = G(P_2)$. Thus  $\gb : a \mapsto \beta_a$ is a bijective map from $G(P_1)$ to $G(P_2)$, named {\it Kim-map}. Let $n_1 < n_2 < \dots < n_g$ and $m_1 < m_2 < \dots < m_g$ be integers such that $G(P_1) = \{ n_1, \dots, n_g \}$ and $G(P_2) = \{ m_1, \dots, m_g \}$. Thus  $\gb : G(P_1) \rightarrow G(P_2)$ is given by $n_i \mapsto m_{\gs(i)}$ for some permutation $\gs$ of the set
$
\N_{\leq g} := \{ 1, \dots, g \}.
$
The graph of $\beta$ is denoted by $\Gamma(P_1,P_2)$, that is,
\begin{equation}\label{graf-beta}
\Gamma(P_1,P_2) := \{(n_i,\beta(n_i)) : i = 1, \ldots, g \}.
\end{equation}
In \cite[Lemma 2]{homma96}, Homma characterized $\Gamma(P_1,P_2)$ in $(G(P_1) \times G(P_2)) \cap H(P_1,P_2)$ as follows.

\begin{lemma} \label{lemma-homma}
Let $\Gamma'$ be a subset of $\Gamma(P_1,P_2)$ in $(G(P_1) \times G(P_2)) \cap H(P_1,P_2)$. If there exists a permutation $\gs$ of $\N_{\leq g} = \{ 1, \dots, g \}$ such that $\Gamma' = \{(n_i, m_{\gs(i)}) : i = 1, \ldots, g\}$, then  $\Gamma' = \Gamma(P_1,P_2)$.
\end{lemma}

Define $R(\gs) = \{  (i,j) \in \mathbb{N}_{\leq g} \times \mathbb{N}_{\leq g} \colon i < j \text{ and } \gs(i) > \gs(j) \}$.
In \cite[Theorem 2.1]{HK}, Homma and Kim described the set $G_0(P_1,P_2)$ as follows.

\begin{theorem} \label{homma-kim}
The set of all pure gaps at $(P_1, P_2)$ is given by
$
G_0(P_1, P_2) = \{  (n_i, m_{\gs(j)})  : \; (i,j) \in  R(\gs) \}.
$
In particular, $\# G_0(P_1, P_2) = \# R(\gs)$.
\end{theorem}

Let $G$ be a divisor on $\mathcal{X}$ defined over $\mathbb{F}_{q}$, and let $D=Q_{1}+\cdots+Q_{m}$ be another divisor on  $\mathcal{X}$  where $Q_{1}, \ldots, Q_{m}$ are distinct $\mathbb{F}_{q}$-rational points, each not belonging to the support of $G$. Then the Goppa (or algebraic geometry) code  $C_{\Omega}(D, G)$ is   the image of the $\F_q$-linear map $\operatorname{res}: \Omega(G-D) \rightarrow \mathbb{F}_{q}^{n}$ defined  by 
$$\eta \mapsto\left(\operatorname{res}_{Q_{1}}(\eta), \operatorname{res}_{Q_{2}}(\eta), \ldots, \operatorname{res}_{Q_{m}}(\eta)\right),$$
where $\Omega(G-D)$ is the $\F_q$-space of differentials $\eta$ on $\xx$ such that $\eta = 0$ or $\mbox{div}(\eta) \geq G-D$. Its  length, dimension and minimum distance are denoted by $\mathrm{len}( C_{\gO}(D,G))$, $\dim( C_{\gO}(D,G))$, and $\mathrm{d}( C_{\gO}(D,G))$, respectively. The dimension of $C_{\Omega}(D, G)$ can be estimated using the Riemann-Roch theorem since it is equal to $i(G-D)-i(G)$, where $i(A) = \ell(A) - \deg(A) + g-1$ is the index of specialty of a divisor $A$.
One of the main features of this code is that its minimum distance satisfies the so-called Goppa bound, namely
$$
\mathrm{d}( C_{\gO}(D,G)) \geq \deg(G) - (2g-2).
$$

Goppa codes with larger minimum distance can be constructed from curves for which the pure gaps in some of their points are known. The following result illustrates this fact.

\begin{theorem} \label{carvalho-torres} \cite[Theorem 3.4]{CT} Let $P_{1}, \ldots, P_{s}, Q_{1}, \ldots, Q_{m}$ be pairwise distinct $\mathbb{F}_{q}$-rational points  on $\mathcal{X}$, and let $\left(a_{1}, \ldots, a_{s}\right)$, $\left(b_{1}, \ldots, b_{s}\right) \in \mathbb{N}_{0}^{s}$ be pure gaps at $\left(P_{1}, \ldots, P_{s}\right)$ with
$a_{i} \leq b_{i}$ for each $i .$ Consider the divisors $D=Q_{1}+\cdots+Q_{m}$ and $G=\sum_{i=1}^{s}\left(a_{i}+b_{i}-1\right) P_{i}$. If every $\left(c_{1}, \ldots, c_{s}\right) \in \mathbb{N}_{0}^{s}$ with $a_{i} \leq c_{i} \leq b_{i}$ for $i \in\{1, \ldots, s\}$ is a pure gap at $\left(P_{1}, \ldots, P_{s}\right)$,
then
$$
\mathrm{d}( C_{\gO}(D,G)) \geq \operatorname{deg}(G)-(2 g-2)+\sum_{i=1}^{s}\left(b_{i}-a_{i}+1\right).
$$
\end{theorem}

\section{Curves with a special triple  $(P_1,P_2,P_3)$} \label{sec3}

\indent

Let $\xx : F(X,Y,Z)=0$ be a projective smooth plane curve of degree $n+1>3$ defined over an algebraically closed field $\K$, and suppose that three distinct lines $\ell_1$, $\ell_2$, and $\ell_3$ cut out on $\xx$ the respective divisors $nP_1+P_2$, $nP_2+P_3$, and $nP_3+P_1$.
Clearly, the three points are distinct, $\ell_1 \cap \ell_2 =\{P_2\}$, $\ell_2 \cap \ell_3 =\{P_3\}$, and $\ell_3 \cap \ell_1 =\{P_1\}$.
In particular, the  three lines are  non-concurrent. Therefore, after a suitable projective transformation, we may assume that the three points are $(1:0:0)$, $(0:1:0)$ and $(0:0:1)$, and that
\begin{equation}\label{axes}
\begin{cases} 
F(X,Y,0)=\alpha XY^n \\
F(X,0,Z)=\beta ZX^n \\
F(0,Y,Z)=\gamma YZ^n,
\end{cases} 
\end{equation}
for some  $\alpha, \beta, \gamma \in \K \setminus \{0\}$. From the three equations in \eqref{axes}, it follows that 
$$
F(X,Y,Z) = \alpha XY^n + \beta ZX^n + \gamma YZ^n + XYZ \cdot G(X,Y,Z),
$$
for some polynomial $G(X,Y,Z)$, which is either zero or   homogeneous  of degree $n-2$.  Since $\K$ is algebraically closed, we may assume  that  $\alpha=\beta= \gamma=1.$
\\

 Hereafter, $\xx$ will represent a nonsingular curve of genus $g=n(n-1)/2$ with homogeneous equation
$$
XY^{n} + YZ^{n} + ZX^{n} + XYZ \cdot G(X,Y,Z) = 0,
$$
and $P_1 = (1:0:0)$, $P_2 = (0:1:0)$, and $P_3 = (0:0:1)$. 
\\

Let $\K(\xx) = \K(x,y)$ denote the function field of $\xx$, where  $x=X/Z$ and $y=Y/Z$.  The tangent lines to $\xx$ at $P_1$, $P_2$ and $P_3$ are $Z=0$, $X=0$ and $Y=0$, respectively. Thus the principal divisor of $x$ and $y$ are given by
\begin{equation}\label{xy}
\begin{cases}
\mbox{div}(x) = -nP_1 + (n-1)P_2 + P_3 \\ 
\mbox{div}(y) = -(n-1)P_1 - P_2 + nP_3.
\end{cases} 
\end{equation}

\begin{theorem} \label{gap1}
For $P \in \{P_1,P_2,P_3\}$, the gap sequence at $P$ is given by 
$$G(P) = \{ (i-1)(n-1)+j \colon 1 \leq i \leq j \leq n-1 \},$$
 and the Weierstrass semigroup $H(P)$ is generated by $S = \{ s(n-1)+1 \colon 1 \leq s \leq n \}$. 
\end{theorem}

\begin{proof}
Let $s \in \{1,\ldots,n\}$. Using the relations in  \eqref{xy}, the following are  obtained
$$
\begin{cases}
\mbox{div} ( y/x^s  ) = ((s-1)n+1) P_1 - (s(n-1)+1) P_2 + (n - s) P_3 \\ 
\mbox{div} ( x^{s-1}/y^s ) = (n-s) P_1 + ((s-1)n+1) P_2 - (s(n-1)+1) P_3 \\
\mbox{div} (xy^{s-1}) = - (s(n-1)+1) P_1 + (n-s) P_2 + ((s-1)n+1) P_3,
\end{cases} 
$$
and then
\begin{equation} \label{2808}
\begin{cases}
\mbox{div}_{\infty} \left( \frac{y}{x^s}  \right) = (s(n-1)+1) P_2 \\
\mbox{div}_{\infty} \left( \frac{x^{s-1}}{y^s}  \right) = (s(n-1)+1) P_3 \\
\mbox{div}_{\infty}(xy^{s-1}) =  (s(n-1)+1)P_1.
\end{cases}
\end{equation}
Given $P \in \{P_1,P_2,P_3\}$, it follows from \eqref{2808} that $S \subset H(P)$. Note that $|\N_0 \setminus \langle S \rangle| = \frac{1}{2} n(n-1)$, where $\langle S \rangle$ is the semigroup generated by $S$. Since $\langle S \rangle \subset H(P)$ and the number of gaps at $P$ is $\frac{1}{2} n(n-1)$, the Weierstrass semigroup $H(P)$ is generated by $S$. Now  since $G(P) = \N_0 \setminus H(P)$, one can easily check that the gap sequence at $P$ is $G(P) = \{ (i-1)(n-1)+j \colon 1 \leq i \leq j \leq n-1 \}$.
\end{proof}

\begin{theorem} \label{basis}
Let $m \in \{1, \ldots, 2g-2\}$, and write $m = d(n-1)+r$ with $0 \leq r \leq n-2$. For $k=1,2,3$, the space $\cL(mP_k)$ has dimension $\frac{d^2-d+2}{2} + \min\{r,d\}$ and basis $B_k$, where
\begin{enumerate}[]
\item $B_1 = \{1\} \cup \{x^iy^j : 1 \leq i \leq d-1 \text{ and } 0 \leq j \leq d-1-i \} \cup \{x^iy^{d-i} : 1 \leq i \leq \min\{d,r\} \}$,
\item $B_2 = \{1\} \cup \left\{\frac{y^i}{x^{i+j}} : 1 \leq i \leq d-1 \text{ and } 0 \leq j \leq d-1-i \right\} \cup \left\{\frac{y^i}{x^d} : 1 \leq i \leq \min\{d,r\} \right\}$,
\item $B_3 = \{1\} \cup \left\{\frac{x^j}{y^{i+j}} : 1 \leq i \leq d-1 \text{ and } 0 \leq j \leq d-1-i \right\} \cup \left\{\frac{x^{d-i}}{y^d} : 1 \leq i \leq \min\{d,r\} \right\}$.
\end{enumerate}
\end{theorem}

\begin{proof}
For $k=1,2,3$, it follows from Theorem \ref{gap1} that $\ell(mP_k) = \frac{d^2-d+2}{2} + \min\{r,d\}$, and a straightforward computation gives $\#B_k = \frac{d^2-d+2}{2} + \min\{r,d\}$. For $1 \leq i \leq d-1$ and $0 \leq j \leq d-1-i$, 

\begin{equation*}
v_{P_k}(x^iy^j)=
\begin{cases}
-(i+j)n+j \geq -m, & \text{ if } k=1 \\
in-i-j \geq 0, & \text{ if } k=2 \\
jn+i \geq 0, & \text{ if  } k=3.
\end{cases}
\end{equation*}
Since the poles of any function $x^uy^v$ with $u,v \in \Z$ lie in $\{P_1,P_2,P_3\}$ (see \eqref{xy}), it follows that $x^iy^j \in \cL(mP_1)$.
Likewise, one can check that $x^iy^{d-i} \in \cL(mP_1)$ for $1 \leq i \leq \min\{d,r\}$.
It is clear that the elements in $B_1$, given by monomials of degree at most $d \leq n-1$, are linearly independent. Therefore, $B_1$ is a basis for $\cL(mP_1)$. The analogous proofs for $B_2$ and $B_3$ follow. 
\end{proof}


\begin{theorem} \label{dim2}
Let $m \in \{1, \ldots, 2g-2\}$, and write $m = d(n-1)+r$ with $0 \leq r \leq n-2$. Then  $\cL(mP_2 - dP_1)$, $\cL(mP_3 - dP_2)$ and $\cL((mP_1 - dP_3)$ have the same dimension $\frac{1}{2}(d-1)(d-2) + \min\{r,d\}$.
\end{theorem}

\begin{proof}
Let $D := mP_2 - dP_1$. By Theorem \ref{basis}, $\ell(mP_2) = \frac{d^2-d+2}{2} + \min\{r,d\}$, and
$$
B_2 = \{1\} \cup \left\{\frac{y^i}{x^{i+j}} : 1 \leq i \leq d-1 \text{ and } 0 \leq j \leq d-1-i \right\} \cup \left\{\frac{y^i}{x^d} : 1 \leq i \leq \min\{d,r\} \right\}
$$
is a basis for $\cL(mP_2)$. Now since $\cL(D) \subset \cL(mP_2)$, an element $f$ belongs to $\cL(D)$ if and only if $f \in \cL(mP_2)$ and $v_{P_1}(f) \geq d$. Note that if $1 \leq i \leq d-1$ and $0 \leq j \leq d-1-i$, then $v_{P_1}(y^i/x^{i+j}) = jn+i$. Hence $v_{P_1}(y^i/x^{i+j}) \geq d$ if and only if $j > 0$. If $\min\{d,r\} > 0$, then $v_{P_1}(y^i/x^d) = (d-i)(n-1) + d \geq d$ for $1 \leq i \leq \min\{d,r\}$. This  and Strict Triangle Inequality  for $v_{P_1}$ imply that
$$
B_2 \setminus \left\{\frac{y^i}{x^i} : 0 \leq i \leq d-1 \right\}
$$
is a basis for $\cL(D)$. Therefore, $\ell(D) = \frac{d^2-d+2}{2} + \min\{r,d\} - d = \frac{1}{2}(d-1)(d-2) + \min\{r,d\}$. The proofs for the other two cases are analogous.
\end{proof}

\begin{theorem} \label{kim-map}
The map
$$
\begin{aligned}
\beta:\{ (i-1)(n-1)+j \colon 1 \leq i \leq j \leq n-1 \} & \longrightarrow \{ (i-1)(n-1)+j \colon 1 \leq i \leq j \leq n-1 \} \\
(i-1)(n-1)+j & \longmapsto (n-j-1)(n-1) + n+i-j-1
\end{aligned}
$$
is a bijection, and $\beta : G(P_1) \longrightarrow G(P_2)$, $\beta : G(P_2) \longrightarrow G(P_3)$, and $\beta^{-1} : G(P_1) \longrightarrow G(P_3)$ are the corresponding Kim-maps.
\end{theorem}

\begin{proof}
It is clear that $\beta$ is a bijection with inverse given by $(i-1)(n-1)+j \longmapsto (j-i)(n-1) + n-i$. Now one can easily check that
\begin{enumerate}[\rm(1)]
\item $\mbox{div}_{\infty} \left( \frac{y^{n-j+i-1}}{x^{n-j-1}} \right) =  ((i-1)(n-1)+j) P_1 + ((n-j-1)(n-1) + n+i-j-1) P_2$
\item $\mbox{div}_{\infty} \left( \frac{1}{x^{i}y^{n-j-1}} \right) =  ((i-1)(n-1)+j) P_2 + ((n-j-1)(n-1) + n+i-j-1) P_3$
\item $\mbox{div}_{\infty} \left( \frac{x^j}{y^{j+1-i}} \right) =  ((i-1)(n-1)+j) P_1 + ((j-i)(n-1) + n-i) P_3$.
\end{enumerate}
Then Lemma \ref{lemma-homma} finishes the proof.
\end{proof}

\section{Pure gaps of $\xx$ at $(P_1,P_2)$} \label{sec4}

\begin{lemma}\label{P1}
The number of pure gaps at $(P_1, P_2)$ is $\frac{1}{3}(g-1)g.$
\end{lemma}

\begin{proof}
Theorem \ref{homma-kim} gives $\# G_0(P_1, P_2) = \# R(\gs)$. Therefore, from Theorem \ref{kim-map}, it suffices to fix $i$ and $j$ with $1 \leq i \leq j \leq n-1$, and count the number of pairs $(I,J)$ with $1 \leq I \leq J \leq n-1$ for which
$$
(i-1)(n-1) + j < (I-1)(n-1) + J
$$
and
$$
(n-j-1)(n-1) + n+i-j-1 > (n-J-1)(n-1) + n+I-J-1.
$$
That is, $(i-I)(n-1) < J-j$ and $(J-j)n > I-i$. One can easily check that these  conditions are equivalent to $i \leq I$  and  $j < J$. Since 
$$
\# \{ (I,J) \colon i \leq I, \; j < J, \; 1 \leq I \leq J \leq n \} = \frac{1}{2}(n-j-1)(j-2i+n+2),
$$
it follows that
$$
\# G_0(P_1, P_2) = \sum_{i=1}^{n-1} \sum_{j=i}^{n-1} \frac{1}{2}(n-j-1)(j-2i+n+2) = \frac{1}{12} (n-2)(n-1)n(n+1)=\frac{1}{3}(g-1)g.
$$

\end{proof}

\begin{corollary}\label{Pp2}
The number of gaps at $(P_1, P_2)$ is $g(g+1)$.
\end{corollary}

\begin{proof}
Note that Theorem \ref{gap1} yields
$$
\sum\limits_{a \in G(P_1)}a = \sum\limits_{b \in G(P_2)}b = \sum\limits_{i=1}^{n-1} \sum\limits_{j=i}^{n-1} ((i-1)(n-1)+j) = \frac{1}{6}n(n-1)(n^2-n+1).
$$
Thus by \cite[Theorem 1]{homma96},
$$
\# G(P_1, P_2) + \# G_0(P_1, P_2) = \sum_{a \in G(P_1)}a + \sum_{b \in G(P_2)}b = \frac{1}{3}n(n-1)(n^2-n+1).
$$
Lemma \ref{P1} gives $\# G_0(P_1, P_2) = \frac{1}{12} (n-2)(n-1)n(n+1)$, and then
$$
\# G(P_1, P_2) = \frac{1}{4}n(n-1)(n^2-n+2)=g(g+1).
$$
\end{proof}

\begin{theorem} \label{Md}
For integers $i,j \geq 1$ with $2 \leq i+j=d \leq n-1$, let $M_d = (in-d)P_1 + j(n-1)P_2$ and $N_d = (i-1)nP_1 + ((j-1)n+i-1)P_2$. Then
$\ell(M_d) = \ell(N_d) = \frac{1}{2}(d-1)(d-2) + i$. 
\end{theorem}

\begin{proof}
Observe that \eqref{xy} yields
$$
\mbox{div} \left( \frac{x^j}{y^d} \right) =  (in-d)P_1 + (jn+i)P_2 - (dn-j)P_3 = M_d - E,
$$
where $E = (d(n-1)+i)P_3 - dP_2$. By Theorem \ref{dim2}, $\ell(M_d) = \ell(E) = \frac{1}{2}(d-1)(d-2) + i$.
\\
Now since $\xx \subset \mathbb{P}^2$ is a smooth curve of degree $n+1$, the divisor cut out by $Z^{n-2}=0$, namely $W := (n-2)nP_1 + (n-2)P_2$, is a canonical divisor. Thus Riemann-Roch  theorem gives
\begin{equation} \label{RRT}
\ell(N_d) = \deg(N_d) +1 -g + \ell(W-N_d).
\end{equation}
It is  claimed that $\ell(W-N_d) = \frac{1}{2}(n-d+1)(n-d+2)$. Note that $W-N_d = (n-i-1)nP_1 - ((j-2)n+i+1)P_2$ and for $j=1$, the divisor $W-N_d = (n-i-1)nP_1 + (n-i-1)P_2$ is cut out on $\xx$ by  $Z^{n-i-1}=0$. Hence, $\ell(W-N_d) = \frac{1}{2}(n-d+1)(n-d+2)$. Now suppose $j \geq 2$. In this case, $\cL(W-N_d) \subset \cL((n-i-1)nP_1)$. By Theorem \ref{basis}, $\ell((n-i-1)nP_1) = \frac{1}{2}(n-i-1)(n-i) + 1$ and
$$
B := \{1\} \cup \{x^uy^v : 1 \leq u \leq n-i-1 \text{ and } 0 \leq v \leq n-i-1-u \}
$$
is a basis for $\cL((n-i-1)nP_1)$. Note that $f \in \cL(W-N_d)$ if and only if $f \in \cL((n-i-1)nP_1)$ and $v_{P_2}(f) \geq (j-2)n+i+1$. Since $v_{P_2}(x^uy^v) = u(n-1) -v$,   $v_{P_2}(x^uy^v) \geq (j-2)n+i+1$ if and only if $u \geq j-1$. Hence
$$
B' := \{x^uy^v : j-1 \leq u \leq n-i-1 \text{ and } 0 \leq v \leq n-i-1-u \}
$$
is a basis for $\cL(W-N_d)$. Therefore, $\ell(W-N_d) = \frac{1}{2}(n-d+1)(n-d+2)$, and then \eqref{RRT} gives $\ell(N_d) = \frac{1}{2}(d-1)(d-2) + i$.
\end{proof}

\begin{proof}[{\bf Proof of Theorem \ref{p1p2}}]
The assertion $\#G_0(P_1,P_2) = \frac{1}{3}(g-1)g$ is already proved in Lemma \ref{P1}.
It follows from Theorem \ref{Md} that the set $S$ given by
$$
\Big\{ (a,b) = \Big((i-1)n+r+1,(j-1)n+i+s \Big) : 2 \leq i+j=d \leq n-1, \ \ r \in \{0,\ldots, n-d-1\}, \ s \in \{0,\ldots, n-d\} \Big\}
$$
is a subset of $G_0(P_1,P_2)$, and the divisors $D=aP_1+bP_2$ and $E = (a-1)P_1+(b-1)P_2$ 
have dimension $\frac{1}{2}(d-1)(d-2) + i$. Note that $S$ is a disjoint union of the sets
$$
S_{ij} = 
\left\{
(a,b) \in \N_0^2 \; \biggr |
\begin{array}{rcccl}
(i-1)n+1        & \hspace{-0.2cm} \leq & \hspace{-0.2cm} a & \hspace{-0.2cm} \leq & \hspace{-0.2cm} in-d  \\
(j-1)n+i & \hspace{-0.2cm} \leq & \hspace{-0.2cm} b & \hspace{-0.2cm} \leq & \hspace{-0.2cm} jn-j
\end{array}
\right\}
$$
with $i,j \geq 1$ and $2 \leq i+j=d \leq n-1$. Since each $S_{ij}$ has $(d-n)(d-n-1)$ elements,
$$
\# S = \# \bigcup_{2 \leq i+j \leq n-1} S_{ij} = \sum_{d=2}^{n-1} (d-1)(d-n)(d-n-1) = \frac{1}{12} (n-2)(n-1)n(n+1) = \frac{1}{3}(g-1)g,
$$
which is the same number of elements in $G_0(P_1,P_2)$. This concludes the proof for $(P_1,P_2)$. By the description of the Kim-maps in Theorem \ref{kim-map}, the analogous proofs for the ordered pairs $(P_2,P_3)$ and $(P_3,P_1)$ apply.
\end{proof}

\section{Pure gaps of $\xx$ at $(P_1,P_2,P_3)$} \label{sec5}

\indent

To characterize the set of pure gaps of $\xx$ at $(P_1,P_2,P_3)$, let us  begin with the following lemmas.

\begin{lemma}\label{dim1} For $i,j,k \in \Z$, and $d:=i+j+k$, the divisor 
$$S_d=(  kn+j ) P_1+(in+k) P_2+(jn+i) P_3$$
is such that $\cL(S_d)$ has dimension
\begin{equation*}
\ell(S_d)=
\begin{cases}
0, & \text{ if } d < 0 \\
\frac{1}{2}(d+2)(d+1), & \text{ if } 0\leq d \leq n-2 \\
(n+1)d-g+1, & \text{ if  } d \geq n-1,
\end{cases}
\end{equation*}
where $g$ is the genus of $\xx$.
\end{lemma}

\begin{proof}
Note that $\mbox{div}(x^iy^j)=-(in+jn-j)P_1 + (in-i-j)P_2 + (jn+i)P_3$, and then
\begin{equation}\label{eq1}
S_d= \mbox{div}(x^iy^j)+ d (n P_1+  P_2)
\end{equation}
gives $S_d\sim d (n P_1+  P_2)$. Thus result  follows from the fact that $n P_1+ P_2$ is the divisor cut out on $\mathcal{X}$ by the line $Z=0$.
\end{proof}

\begin{lemma}\label{lema2} Notation as in Lemma \ref{dim1}. If $0\leq d \leq n-2$ and
$d+e=n-2$, then
\begin{equation*}
\ell(S_d+e(P_1+P_2+P_3))=\frac{1}{2}(d+2)(d+1).
\end{equation*}
\end{lemma}
\begin{proof}
Let $i_1,i_2,i_3 \in \Z$ be such that $i_1+i_2+i_3=e$. Since $d+e=n-2$,
it follows that $S_{d+e}$ is a canonical divisor. Thus for $D:=S_d+e(P_1+P_2+P_3)$,

$$S_{d+e}-D=S_e-e(P_1+P_2+P_3)= \mbox{div}(x^{i_1}y^{j_1})+ e ((n-1) P_1- P_3),$$
where the second equality above follows from \eqref{eq1}. Therefore, by Riemann-Roch theorem 
$$\ell(D)=\deg D+1-g+\ell(e ((n-1) P_1-P_3)).$$
From Theorem \ref{dim2},  $\ell(e ((n-1) P_1-P_3))=\frac{1}{2}(e-1)(e-2)$, and then
$$\ell(D)=d(n+1)+3e+1-\frac{1}{2}n(n-1)+\frac{1}{2}(e-1)(e-2) = \frac{1}{2}(d+2)(d+1).$$
\end{proof}

\begin{lemma} \label{GG}
If $(a,b,c) \in G_0(P_1,P_2,P_3)$ then none of $a$, $b$ or $c$ is divisible by $n-1$. 
\end{lemma}

\begin{proof}
Suppose that $n-1$ divides $a$. From Theorems \ref{thm_faces} and \ref{gap1}, it follows that $a = t(n-1)$ for some $t \in \{ 1, \dots, n-1 \}$. Note that Theorem \ref{kim-map} gives $\gb(i(n-1)) = i$, for $i=1,\ldots,n-1$, and \eqref{graf-beta} yields
\begin{equation} \label{29-10}
(i(n-1),i,0) \in H(P_1,P_2,P_3) \text{ for } i=1,\ldots,n-1.
\end{equation}
Given that $(t(n-1),b,c) \in G_0(P_1,P_2,P_3)$ and $(t(n-1),t,0) \in H(P_1,P_2,P_3)$, it follows from Theorem \ref{thm_faces} that $b < t$ and $(b(n-1),b,0) \in G(P_1,P_2,P_3)$, which contradicts \eqref{29-10}. Similarly, using that $(0,i(n-1),i),(i,0,i(n-1)) \in H(P_1,P_2,P_3)$ for $i=1,\ldots,n-1$, the same proof applies for coordinates $b$ and $c$.
\end{proof}

\begin{proof}[{\bf Proof of Theorem \ref{cor}}]
It follows immediately from Lemmas \ref{dim1} and \ref{lema2} that $G_0(P_1,P_2,P_3)$ contains the elements
\begin{equation} \label{abc}
(a,b,c)=\Big(kn+j+r+1,in+k+s+1, jn+i+t+1\Big)
\end{equation}
with $i,j,k \geq 0$, $i+j+k=d \leq n-3$ and $r,s,t \in \{0,\ldots, n-3-d\}$. Furthermore, for $(a,b,c)$ and $d$ as above, the Riemann-Roch spaces $\cL(aP_1+bP_2+cP_3)$ and $\cL((a-1)P_1+(b-1)P_2+(c-1)P_3)$ have dimension $\frac{1}{2}(d+1)(d+2)$.
\\

\noindent
Now the fact  that any $(a,b,c) \in G_0(P_1,P_2,P_3)$ can be written as in \eqref{abc} follow. Indeed, from Theorem \ref{kim-map} and \eqref{graf-beta}, 
\begin{enumerate}[\rm(1)]
\item $(a,\gb(a),0), (\gb^{-1}(b),b,0) \in H(P_1,P_2,P_3)$
\item $(0,b,\gb(b)), (0,\gb^{-1}(c),c) \in H(P_1,P_2,P_3)$
\item $(a,0,\gb^{-1}(a)), (\gb(c),0,c) \in H(P_1,P_2,P_3)$.
\end{enumerate}
Thus Theorem \ref{thm_faces} entails
\begin{itemize}
\item[(i)] $b < \gb(a)$ and $a < \gb^{-1}(b)$
\item[(ii)] $c < \gb(b)$ and $b < \gb^{-1}(c)$
\item[(iii)] $a < \gb(c)$ and $c < \gb^{-1}(a)$.
\end{itemize}
By \eqref{gap1} and Lemma \ref{GG}, there exist $i,j,k \in \{ 0, \dots, n-3 \}$ such that
$$
\begin{cases} 
a = kn + a_0 \\ 
b = in + b_0 \\
c = jn + c_0
\end{cases}
\text{ with } \quad
\begin{cases}
1 \leq a_0 \leq (n-2) - k \\ 
1 \leq b_0 \leq (n-2) - i \\
1 \leq c_0 \leq (n-2) - j
\end{cases}.
$$
Theorem \ref{kim-map} gives
$$
\begin{cases} 
\beta(a) = (n-a_0-k-1)(n-1) + n -a_0 \\ 
\beta(b) = (n-b_0-i-1)(n-1) + n -b_0 \\
\beta(c) = (n-c_0-j-1)(n-1) + n -c_0
\end{cases}
\text{ and } \quad
\begin{cases}
\beta^{-1}(a) = (a_0-1)(n-1) + n-k-1 \\ 
\beta^{-1}(b) = (b_0-1)(n-1) + n-i-1 \\
\beta^{-1}(c) = (c_0-1)(n-1) + n-j-1.
\end{cases}
$$
Let $d := i+j+k$, $r := a_0-j-1$, $s := b_0-k-1$ and $t := c_0-i-1$. Thus
$$
\begin{cases}
a = kn + j + r + 1 \\ 
b = in + k + s + 1 \\
c = jn + i + t + 1.
\end{cases} 
$$
Note that $c < \gb^{-1}(a)$ gives $j(n-1) + (c_0+j) \leq (a_0-1)(n-1) + (n-k-2)$. Since $c_0+j, n-k-2 \in \{0, \ldots, n-2\}$, $j \leq a_0-1$, and then $r \geq 0$. Likewise, $a < \gb^{-1}(b)$ and $b < \gb^{-1}(c)$ give $s,t \geq 0$. Using $b < \gb(a)$,  it follows that $i(n-1) + (b_0+i) \leq (n-a_0-k-1)(n-1) + (n-a_0-1)$. Since $b_0+i, n-a_0-1 \in \{0, \ldots, n-2\}$,  either $i < n-a_0-k-1$ or $i = n-a_0-k-1$ with $b_0+i \leq n-a_0-1$. The latter case gives $b_0 \leq k$, which contradicts $b_0 = s+k+1$ and $s \geq 0$. Hence  $i < n-a_0-k-1$. Using  $a_0 = r+j+1$, it follows that  $r \leq n-3-d$. In particular, $d \leq n-3$. Analogously,   $s,t \leq n-3-d$.  This proves that $(a,b,c)$ is of the form as in \eqref{abc}. \\

Finally, to compute the cardinality of $G_0(P_1,P_2,P_3)$, note that for each $d \in \{ 0, \dots, n-3 \}$, the number of triples $(r,s,t)$ with $r,s,t \in \{0,\ldots, n-3-d\}$ is $(n-d-2)^3$, and the number of triples $(i,j,k) \in \N^3$ with $i+j+k = d$ is
$
\frac{1}{2}(d+1)(d+2).
$
Therefore, $\#G_0(P_1,P_2,P_3)$ is given by
$$
\sum_{d=0}^{n-3} \frac{1}{2}(d+1)(d+2)(n-(d+2))^3 = \frac{1}{120}(n-2)(n-1)n(n+1)(n^2-n-1)=\frac{1}{30}(g-1)g(2g-1).
$$
\end{proof}

\section{Goppa codes supported on $(P_1,P_2)$ and $(P_1,P_2,P_3)$} \label{sec6}

\indent

In this section, we  show how the characterization of  pure gaps presented in the previous sections can   enable  us  to construct algebraic geometry codes with good parameters, having minimum distance better than the Goppa bound. The codes are constructed from nonsingular curves $\xx$ defined over $\F_q$ given by
\begin{equation}\label{c-geral}
XY^{n} + YZ^{n} + ZX^{n} + XYZ \cdot G(X,Y,Z) = 0.
\end{equation}
Let $P_1 = (1:0:0)$, $P_2 = (0:1:0)$ and $P_3 = (0:0:1)$. Goppa codes supported on $\{P_1,P_2,P_3\}$ will be constructed, as follows.  \footnote{For further  details regarding Goppa codes, see \cite[Chapter 2]{HS}.} 
\\

Let us  begin with Goppa codes supported on $(P_1,P_2)$. For $i,j\geq 1$ and $i+j=d\leq n-1$, it follows from Theorem \ref{p1p2} that the pairs
$$
(\ga_1, \ga_2) = ((i-1)n+1, (j-1)n+i) \ \text{and} \ (\gb_1,\gb_2) = (in-i-j, jn-j)
$$
are pure gaps of $\xx$ at $(P_1,P_2)$. In addition,  for $s=1,2$,  the  integers  $t_s$ with $\gb_s \leq t_s \leq \ga_s$ are such that  $(t_1,t_2) \in G_0(P_1,P_2)$  and $\ell(t_1P_1+t_2P_2)=\frac{1}{2}(d-1)(d-2)+i$.  Let us choose $m$ distinct points  $Q_1,\dots,Q_m \in \xx(\F_q) \setminus \{P_1,P_2\}$ and consider the divisors $D = Q_1+\cdots+Q_m$ and $F_1 = (\ga_1 + \gb_1 - 1)P_1 + (\ga_2 + \gb_2 - 1)P_2$. Let $C_{\gO}(D,F_1)$ be the corresponding Goppa code. With this notation,  the following holds.

\begin{theorem} \label{code2}
Let $\xx$ be the curve in \eqref{c-geral}. If $\frac{n+2}{2} \leq i+j \leq n-1$ and $m \geq 2n^2-4n-2$, then $C_{\gO}(D,F_1)$ is a code with parameters
$$
\begin{cases}
\mathrm{len}( C_{\gO}(D,F_1)) = m \\
\dim( C_{\gO}(D,F_1))= m + \frac{1}{2}n^2 + \left( \frac{3}{2}-2(i+j) \right)n+2j \\
d( C_{\gO}(D,F_1))\ge  \left( 2(i+j)+1 \right)n -n^2  -2i -4j +2.
\end{cases}
$$
\end{theorem}

\begin{proof}
The length of $C_{\gO}(D,F_1)$ is $\deg(D) = m$. Since $\frac{n+2}{2} \leq i+j$ implies $\deg(F_1) > 2g-2$, where $g = n(n-1)/2$ is the genus of $\xx$, it follows that  $i(F_1)= \ell(F_1) - \deg(F_1) + g - 1 = 0$. Moreover, $m \geq 2n^2-4n-2$ implies $\deg(F_1-D)<0$, and then $\ell(F_1-D) = 0$. Hence
\begin{eqnarray*}
\dim( C_{\gO}(D,F_1))\hspace{-0.2cm} & = & \hspace{-0.2cm} i(F_1-D) - i(F_1) = i(F_1-D) = \ell(F_1-D) - \deg(F_1-D) +g -1 \\
 & = & \hspace{-0.2cm} \deg(D) - \deg(F_1) +g-1 = m + \frac{1}{2}n^2 + \left( \frac{3}{2}-2(i+j) \right)n+2j.
\end{eqnarray*}
By Theorem \ref{carvalho-torres}, $d( C_{\gO}(D,F_1))\ge \deg(F_1) - (2g-2) + (\gb_1 - \ga_1) + (\gb_2 - \ga_2) + 2$. Therefore,
$
d( C_{\gO}(D,F_1))\ge \left( 2(i+j)+1 \right)n -n^2  -2i -4j +2.
$
\end{proof}

Now let us construct Goppa codes supported on $(P_1,P_2,P_3)$. For $i,j,k \geq 0$ with $i+j+k=d \leq n-3$, it follows from Theorem \ref{cor} that 
$$(n_1,n_2,n_3) = (kn+j+1,in+k+1,jn+i+1)$$ and 
$$(p_1,p_2,p_3) = \Big((k+1)n-k-i-2,(i+1)n-i-j-2,(j+1)n-j-k-2\Big)$$ are pure gaps at $(P_1,P_2,P_3)$. In addition,  for $s=1,2,3$,  the  integers  $t_s$ with $n_s \leq t_s \leq p_s$ are such that  $(t_1,t_2,t_3) \in G_0(P_1,P_2,P_3)$ and $\ell(t_1P_1+t_2P_2+t_3P_3)=\frac{1}{2}(d+1)(d+2)$.
Let us choose m distinct points  $Q_1,\dots,Q_m\in \xx(\F_q) \setminus \{P_1,P_2,P_3\}$ and consider the divisors $D= Q_1+\cdots+Q_m$ and
$
F_2 = \sum_{s=1}^{3} (n_s+p_s-1)P_s.
$
Let $C_{\gO}(D,F_2)$ be the corresponding Goppa code. Note that $\deg (F_2)=(2d+3)n-d-6 $. From this,  the  same proof of Theorem \ref{code2} gives  the following.

\begin{theorem} \label{code3}
Let $\xx$ be the curve in \eqref{c-geral}. If $\frac{(n-2)^2}{2n-1} < i+j+k=d \leq n-3$ and $m \geq 2n^2-4n-2$, $C_{\gO}(D,F_2)$ is a  Goppa code  with parameters
$$
\begin{cases}
\mathrm{len}( C_{\gO}(D,F_2)) = m \\
\dim( C_{\gO}(D,F_2))= m + \frac{1}{2}n^2 -(2d+\frac{7}{2})n +d+5 \\
d( C_{\gO}(D,F_2))\ge (2d+7)n -n^2-4d-10.
\end{cases}
$$
\end{theorem}
%

\begin{remark}
Note that the lower bounds for the minimum distance in Theorems \ref{code2} and \ref{code3} are better than the Goppa bound
$d( C_{\gO}(D,F))\geq \deg(F) - (2g-2)$.
\end{remark}

The previous results can be used  to construct several examples of Goppa codes with good parameters. A few of them are presented bellow.
\begin{example}
Consider the nonsingular curves $\xx : XY^4 + YZ^4 + ZX^4 + XYZ \cdot G(X,Y,Z)  = 0$  over $\F_q$, where $G(X,Y,Z)$ is homogeneous of degree $2$. For $i=2$ and $j=1$ in Theorem \ref{code2},  the following table  presents constructive   $[n,k,d]$-codes with parameters matching current records of  the database in \cite{MinT}.
$$
\begingroup
\setlength{\tabcolsep}{10pt} 
\renewcommand{\arraystretch}{1.5} 
\begin{array}{|c|c|c|c|}
\hline G(X,Y,Z) & q & \# \xx(\F_q) & [n,k,d] \\
\hline \hline X^2+XY-Y^2-YZ & 27 & 59 & [ 57, 49, d \geq 6 ] \\
\hline X^2 + Y^2 & 16 & 39 & [ 37, 29, d \geq 6 ] \\
\hline XY + Y^2 + XZ + YZ & 128 & 199 & [ 197, 189, d \geq 6 ] \\
\hline 2X^2 + XY + Y^2 + XZ + 2Z^2 & 81 & 145 & [ 143, 135, d \geq 6 ] \\
\hline 2X^2 + 2Y^2 + 3XZ + 6YZ + 2Z^2 & 49 & 100 & [ 98, 90, d \geq 6 ] \\
\hline
\end{array}
\endgroup
$$
\end{example}

The results can also attain  new records for constructive  codes  in  \cite{MinT}, such as the following.

\begin{example}
The nonsingular $\xx$ curve over $\F_{49}$ defined by
$$XY^5+ YZ^5 + ZX^5  + XYZ(5Y^3 + 4Y^2X + 4YX^2 + 6X^3 + 5Y^2Z + 3X^2Z + 3XZ^2 + 2Z^3)=0$$
has $115$   $\F_{49}$-rational points. From Theorem \ref{code2},  this curve gives rise to new constructive codes  over $\F_{49}$ of parameters $[ 113, 95,d], [ 112, 94,d], [ 111, 93,d],[ 110, 92,d],[ 109, 91,d],[ 108, 90,d],$ and $[ 107, 89,d]$, with $d \geq 12$. According to  \cite{MinT}, the previous record for the minimum distance of such codes was $d=11$.
\end{example}

Additional  codes can be constructed from general  curves \eqref{c-geral} with many rational points.  The following   illustrate some general cases.

\begin{example}
Consider the Hurwitz curve $\xx_q$ with homogeneous equation $XY^{q+1} + YZ^{q+1} + ZX^{q+1} = 0$ defined over $\F_{q^3}$. The number of $\F_{q^3}$-rational points of $\xx_q$ is equal to
$
2q^3 + 1 + (1 - \varepsilon_q)(q^2+q+1),
$
where $\varepsilon_q\in \{0,1,2\}$ is such that $q+1\equiv \varepsilon_q  \bmod 3$ (see \cite[Theorem 3.6]{Pel}). For $(q+3)/2 \leq i+j \leq q$,  Theorem \ref{code2} yields a code  $C_{\gO}(D,F_1)$  over $\F_{q^3}$  with parameters
%
%
$$
\begin{cases}
\mathrm{len}( C_{\gO}(D,F_1)) = 2q^3 - 1 + (1 - \varepsilon_q)(q^2+q+1)\\
\dim( C_{\gO}(D,F_1))= 2q^3 + \frac{1}{2}q^2 - \left( 2(i+j) - \frac{5}{2} \right)q -2i+1 + (1 - \varepsilon_q)(q^2+q+1) \\
d( C_{\gO}(D,F_1))\ge  2\left( i+j -\frac{1}{2} \right)q -q^2 -2j +2.
\end{cases}
$$
\end{example}

\begin{example}
Let $\cC$ be the  curve  $XY^{q} + YZ^{q} + ZX^{q} = 0$ defined over $\F_{q^3}$. By \cite[Proposition 4.6]{CKT}, $\cC$ is $\F_{q^3}$-isomorphic to the Hermitian curve $X^{q+1}+Y^{q+1}+Z^{q+1}=0.$ Therefore, $\cC$ is $\F_{q^{6}}$-maximal, that is, $ \# \cC(\F_{q^6}) =q^{6}+q^{5}-q^{4}+1$.  For $\frac{(q-2)^2}{2q-1} < d \leq q-3$, it follows from  Theorem \ref{code3}  that  $\cC$ gives rise to a code over $\F_{q^{6}}$ with parameters
$$
\begin{cases}
\mathrm{len}( C_{\gO}(D,F_2)) = q^{6}+q^{5}-q^{4}-2\\
\dim( C_{\gO}(D,F_2))= q^{6} + q^{5} - q^{4} + \frac{1}{2}q^2 -(2d+\frac{7}{2})q + d + 3 \\
d( C_{\gO}(D,F_2))\ge (2d+7)q - q^2 -4d -10.
\end{cases}
$$
\end{example}

\bibliographystyle{amsplain}

\vspace{0,5cm}\noindent {\em Authors' addresses}:

\vspace{0.2cm}\noindent Herivelto Borges \\ Instituto de Ciências Matemáticas e de Computação\\  Universidade de São Paulo\\ Avenida Trabalhador São-carlense, 400\\ 13566-590 - São Carlos - SP (Brazil).\\E--mail: {\tt hborges@icmc.usp.br } \\

\vspace{0.2cm}\noindent Gregory  Cunha\\ Instituto de Matemática e Estatística \\  Universidade Federal de Goiás  \\ Campus Samambaia - Rua Jacarandá - Chácaras Califórnia \\ 74001-970 - Goiânia - GO (Brazil).\\ E--mail: {\tt gregoryduran@ufg.br } \\

\vspace{0.2cm}\noindent {\em Acknowledgment}:

\vspace{0.2cm}\noindent The first  author thanks the financial support of  CNPq-Brazil (grants 421440/2016-3  and 311572/2019-7).

\vspace{0.2cm}\noindent The second author thanks the financial support of FAPESP-Brazil (grant 2018/01548-3).

\end{document}